\def\draft{n}
\documentclass{amsart}
\usepackage{fullpage,amssymb,epic,eepic,epsfig,amsmath}

\theoremstyle{plain}
\newtheorem{theorem}{Theorem}

\newtheorem{proposition}{Proposition}[section]
\newtheorem{lemma}[proposition]{Lemma}
\newtheorem{ttheorem}[proposition]{Theorem}
\newtheorem{corollary}[proposition]{Corollary}
\theoremstyle{definition}
\newtheorem{definition}[proposition]{Definition}
\newtheorem{fact}[proposition]{Fact}

\theoremstyle{remark}

\newtheorem{remark}[proposition]{Remark}

\def\printname#1{
	\if\draft y
		\smash{\makebox[0pt]{\hspace{-0.5in}
			\raisebox{8pt}{\tt\tiny #1}}}
	\fi
}

\def\mapright#1{\smash{\mathop{\longrightarrow}\limits^{#1}}}
\def\pNor{c_{\mathrm{match}}}
\def\pev{c_{\mathrm{Ising}}}
\def\cdet{c_{\mathrm{det}}}
\def\cform{c_{\mathrm{form}}}
\def\BD{{\mathbb D}}

\def\sign{\mathrm{sign}}
\def\cross{\mathrm{cr}}

 \newlength{\standardunitlength}
\catcode`\@=11
\long\def\@makecaption#1#2{%
     \vskip 10pt

\setbox\@tempboxa\hbox{%\ifvoid\tinybox\else\box\tinybox\fi
       \small\sf{\bfcaptionfont #1. }\ignorespaces #2}%
     \ifdim \wd\@tempboxa >\captionwidth {%
         \rightskip=\@captionmargin\leftskip=\@captionmargin
         \unhbox\@tempboxa\par}%
       \else
         \hbox to\hsize{\hfil\box\@tempboxa\hfil}%
     \fi}
\font\bfcaptionfont=cmssbx10 scaled \magstephalf
\newdimen\@captionmargin\@captionmargin=2\parindent
\newdimen\captionwidth\captionwidth=\hsize
\catcode`\@=12

\def\lbl#1{\label{#1}\printname{#1}}

\def\BZ{\mathbb Z}
\def\BQ{\mathbb Q}

\def\BF{\mathbb F}

\def\C{\mathcal C}

\def\E{\mathcal E}
\def\P{\mathcal P}
\def\E{\mathcal E}

\def\Pfaf{\operatorname{Pfaf}}
\def\Arf{\mathrm{Arf}}
\def\Hom{\mathrm{Hom}}

\begin{document}

\title[Optimality of the Arf Invariant Formula]{On the optimality of
  the Arf Invariant Formula for graph polynomials}

\author{Martin Loebl}
\address{Department of Applied Mathematics and 
Institute of Theoretical Computer Science (ITI) \\
Charles University \\
Malostranske n. 25 \\
118 00 Praha 1 \\
Czech Republic.}
\email{loebl@kam.mff.cuni.cz}

\author{Gregor Masbaum}
\address{Institut de Math{\'e}matiques de Jussieu (UMR 7586 du CNRS)\\
Universit{\'e} Paris Diderot (Paris 7) \\
Case 7012 - Site Chevaleret\\
75205 Paris Cedex 13\\
France }
\email{masbaum@math.jussieu.fr}
\urladdr{www.math.jussieu.fr/\textasciitilde masbaum/}

\date{Revised version: 28 May 2010 (First version: 2 August  2009)
\newline 2010 {\em Mathematics Classification.} Primary 05C31 Secondary 57M15
}

\begin{abstract}
We prove optimality of the Arf invariant formula for the generating
function of even subgraphs, or, equivalently, the Ising partition
function, of a graph.
\end{abstract}

\maketitle

\tableofcontents

\section{Introduction}
\lbl{sec.intro}
Let $G= (V(G),E(G))$ be a finite unoriented  graph (loop-edges and multiple
edges are allowed). We say that $E'\subset E(G)$ is
{\em even} if the graph $(V(G),E')$ has even degree (possibly
zero) at each vertex. By abuse of language, $E'$ is also called an
{\em even subgraph} of $G$. We say that $M\subset E(G)$ is a {\em perfect matching} if the
graph $(V(G),M)$ has degree one at each vertex.  Let $\mathcal{E}(G)$
denote the set of all even subgraphs of $G$, and let $\mathcal {P}(G)$
denote the set of all perfect matchings of $G$. 

We assume that an indeterminate $x_e$ is associated with each edge
$e$, and define the generating polynomials for even sets and for
perfect matchings,  $\E_G$ and $\P_G$, in $
\BZ[(x_e)_{e\in E(G)}]$,  as follows: 
$$\E_G(x)= \sum_{E'\in\mathcal{E}(G) }\,\,\prod_{e\in E'}x_e~,$$
$$\P_G(x)= \sum_{M\in\mathcal{P}(G) }\,\,\prod_{e\in M}x_e~.$$
 Knowing the polynomial $\E_G$ is equivalent to knowing the partition
 function 
 $Z^{\mathrm{Ising}}_G$
of the Ising model on
the graph $G$.
This is explained later in the introduction.

Assume the vertices of $G$ are numbered
from $1$ to $n$. If $D$ is an orientation of $G$, we denote by $A(G,D)$ the {\em skew-symmetric adjacency
matrix} of $D$ defined as follows: The diagonal entries of  $A(G,D)$ are zero,
and the off-diagonal entries are 
$$A(G,D)_{ij}=\sum \pm x_e~,$$ where the sum is over all edges $e$
connecting vertices $i$ and $j$, and the sign in front of $x_e$ is $1$ if $e$ is
oriented from $i$ to $j$ in the orientation $D$, and $-1$
otherwise. As is well-known, the Pfaffian
of this matrix counts perfect matchings of $G$ {\em with signs:} 
$$\Pfaf A(G,D) = \sum_{M\in \mathcal{P}(G)}\sign(M,D) \prod_{e\in
  M}x_e~,$$ where $\sign(M,D)=\pm 1$. We use this as the definition of the sign of a perfect
matching $M$ with respect to an orientation $D$.

We denote the polynomial 
 $\Pfaf A(G,D)\in \BZ[(x_e)_{e\in E(G)}]$ 
by
$F_D(x)$ and call it the {\em
Pfaffian} associated to the orientation $D$. The following result is
well-known.

\begin{theorem}[Kasteleyn \cite{K}, Galluccio-Loebl \cite{GL1},
  Tesler \cite{T}, Cimasoni-Reshetikhin \cite{CR}] \lbl{Pfaff-formula} If $G$
embeds into an orientable surface of genus $g$, then there exist $4^g$
orientations $D_i$ ($i=1,\ldots, 4^g$) of $G$ such that the perfect matching
polynomial $\P_G(x)$ can be
expressed as a linear combination  of the Pfaffian polynomials
$F_{D_i}(x)$.
\end{theorem}

The explicit expression for $\P_G(x)$ will be given in
Theorem~\ref{arfin2}.  We call it the {\em Arf-invariant formula}, as it is
based on a property of the Arf invariant of quadratic forms in
characteristic two. As far as we know, the relationship with the Arf
invariant was first observed in \cite{CR}. 

Let $\pNor(G)$ be the minimal number of orientations $D_i$ of $G$ so
that $\P_G(x)$ is a linear combination of the Pfaffian polynomials
$F_{D_i}(x)$. We think of $\pNor(G)$ as a kind of complexity of
the graph $G$. Since every graph embeds into some surface, $\pNor(G)$
is finite. Norine \cite{N} conjectured that $\pNor(G)$ is always a
power of $4$. He also showed that $\pNor(G)$ cannot be equal to $2,
3,$ or $5$. However, Miranda and Lucchesi 
\cite{ML} recently disproved Norine's conjecture by exhibiting a graph
$G$ with  $\pNor(G)=6$.

The main result of the present paper is that, contrary to the case of
perfect matchings,  an analogue of Norine's
conjecture is true for the even subgraph
polynomial $\E_G(x)$ (or, equivalently, the Ising partition function
of $G$). To explain the statement, we first need to recall how 
the Arf-invariant formula
for $\P_G$ 
can be used to obtain a similar formula for $\E_G$, using the
following
slight modification of a 
construction of Fisher. (This formula for $\E_G$  follows from
\cite[Theorem 2.3]{GL2}.)
Although the construction may seem a little bit un-natural at first
sight, it is 
justified by Proposition~\ref{thm.delta} and Theorem~\ref{Pf2}
below. We'll briefly comment on a different construction by Kasteleyn
in Remark~\ref{newremark}.  
 
\begin{definition}
\lbl{def.delta} (Fisher \cite{Fi})
Let $G$ be a graph. Let $\sigma=(\sigma_v)_{v\in V(G)}$ be a choice, for every vertex $v$, of
a linear ordering of the half-edges incident with $v$. The {\em blow-up}, or
{\em ${\Delta}$-extension}, 
of $(G,\sigma)$ is the graph $G_{}^\sigma$ obtained by performing the following operation
one by one for each vertex $v$.
Assume first that no edge incident with $v$ is a loop-edge. Then
$\sigma_v$ is the same as a
linear ordering of the edges incident with $v$.
Let $e_1, \ldots, e_d$ be 
this linear ordering
and let $e_i= vu_i$, $i=1, \ldots, d$. We delete
the vertex $v$
and replace it with a path consisting of $6d$ new vertices $v_1, \ldots, v_{6d}$
and edges $v_iv_{i+1}$, $i=1, \ldots ,6d-1$. To this path, we add
edges $v_{3j-2}v_{3j}$, $j=1, \ldots, 2d$. 
Finally we add  edges $v_{6i-4}u_i$ corresponding to the original
edges $e_1, \ldots, e_d$.
This definition can be extended naturally to the case where there are
loop-edges, using that $\sigma_v$ is a linear ordering on the set of half-edges
incident with $v$. 
\end{definition}

The subgraph of $G^\sigma$
spanned by the
$6d$ vertices $v_1,\ldots, v_d$ that replaced a vertex $v$ of the
original graph will be called a {\em gadget} and denoted by
$\Gamma_v$. The edges of $G^\sigma$ which do not belong to a gadget are
in natural bijection with the edges of $G$. By abuse of notation, we will identify an edge of $G$ with the
corresponding edge of $G_{}^\sigma$.  Thus $E(G^\sigma)$ is the disjoint union of $E(G)$ and the
various $E(\Gamma_v)$ ($v\in V(G)$). 
 
It is important to note that different choices of linear orderings
$\sigma_v$ at the vertices of $G$  may lead to non-isomorphic graphs $G^\sigma$.
Nevertheless, one always has the following   
\begin{proposition}[Fisher \cite{Fi}]
\lbl{thm.delta}
There is a natural bijection between the set of even subsets of $G$ and the set
of perfect matchings of $G_{}^\sigma$. More precisely, every even
set $E'\subset E(G)$ uniquely extends to a perfect matching $M\subset
E(G_{}^\sigma)$, and every perfect matching of
$G_{}^\sigma$ arises (exactly once) in this way.
\end{proposition}

It follows that if we set the indeterminates associated to the edges
of the gadgets equal to one in $\P_{G_{}^\sigma}$, we get
the even subgraph polynomial of our original graph $G$: 
\begin{equation}\lbl{E=P}\E_G\ =\ \P_{G_{}^\sigma}\Big|
  _{\textstyle{x_e=1 \ \forall  e\in
  E(G_{}^\sigma) \backslash E(G)}}
\end{equation}

If $D$ is an orientation of $G_{}^\sigma$, we define 
$$F_D^\sigma(x) = \Pfaf \Big(A(G_{}^\sigma,D)\Big| _{\textstyle{x_e=1 \ \forall e\in
  E(G_{}^\sigma) \backslash E(G)}}\Big)$$ 
  Any
polynomial obtained in this way will be called a {\em
  $\sigma$-projected Pfaffian}. Note that $F_D^\sigma(x)$ is a polynomial in the
indeterminates associated to the edges of the original graph $G$.

\begin{remark} Here, as before, we need to
choose an ordering of the
vertices of $G^\sigma$ to define the adjacency matrix. We may, of
course, take the ordering induced in the obvious way from an ordering
of the vertices of $G$. In any case, permuting the
ordering will only affect the sign of $F_D^\sigma(x)$: it gets
multiplied by the sign of the permutation.
\end{remark}

Now assume $G$ is 
embedded into an orientable surface $\Sigma$ of genus $g$. It is not
hard to see that we can choose $\sigma$ in such a way that $G^\sigma$
also embeds into $\Sigma$. 
In view of (\ref{E=P}), Theorem~\ref{Pfaff-formula} implies the
following result for $\E_G$:  

\begin{theorem} [Galluccio-Loebl \cite{GL2}] \lbl{Pf2} If $G$
embeds into an orientable surface of genus $g$, then for an
appropriate choice of blow-up  $G^\sigma$, there exist $4^g$
orientations $D_i$ ($i=1,\ldots, 4^g$) of $G^\sigma$ such that the
even subgraph
polynomial $\E_G(x)$ can be
expressed as a linear combination  of the $\sigma$-projected Pfaffians 
$F_{D_i}^\sigma(x)$.
\end{theorem}

\begin{remark}\lbl{newremark} As was pointed out to us by the referee,
  a similar result can also be obtained using a different and in some sense
  more natural blow-up construction discussed by Kasteleyn in
  \cite[p.102-103]{K}, where every vertex of the original
  graph $G$ is replaced by an even clique. In Kasteleyn's construction
  the
  correspondence between perfect matchings of the blown-up graph and
  even subgraphs of the original graph is not one-to-one, but
  many-to-one; however, with an appropriate choice
  of orientation of the edges of the clique, all but one of the
  perfect matchings corresponding to a fixed even subgraph cancel out
  when signs are taken into account. The drawback of Kasteleyn's
  construction is that the blown-up graph can in general not be
  embedded on the same surface as the original graph. Although the
  complications resulting from this problem can be dealt with, it is
  more appropriate for our purposes to use  Fisher's construction where
  this problem does not arise.
\end{remark}

It turns out that one can always choose the orientations $D_i$ in 
Theorem~\ref{Pf2}
in such a way that the induced orientation on every gadget
$\Gamma_v$ 
is independent of $i$.  (We will explain why this is so in
Section~\ref{sec3}, see Corollary~\ref{Pf2delta}.) This
motivates the following definition.

\begin{definition} Let $\Delta=(\Delta_v)_{v\in V(G)}$ be a choice of
  orientations of the gadgets $\Gamma_v$. An orientation $D$
  of $G^\sigma$ is called {\em $\Delta$-admissible} if $D$ restricts to
  $\Delta_v$ on every gadget $\Gamma_v$.
\end{definition}

Note that once $\Delta$ has been fixed, the set of $\Delta$-admissible
orientations of $G^\sigma$ is in natural bijection with the set of
orientations of the original graph $G$.

We now come to the main result of the paper, which is a lower bound
for the number of orientations needed in the above expression for
$\E_G$, provided we assume that the orientations in question are
$\Delta$-admissible.

\begin{theorem} [Main Theorem] \lbl{MT} Let $G$ be a graph. Choose a blow-up   $G^\sigma$ and an orientation $\Delta$ of the gadgets that
replaced the vertices of $G$ in $G^\sigma$. Let $c_{\sigma,\Delta}(G)$ be the
minimal cardinality of a set of $\Delta$-admissible orientations $D_i$
of $G^\sigma$ such that the even subgraph polynomial $\E_G$ is a linear
  combination of the $\sigma$-projected Pfaffians
  $F^\sigma_{D_i}$. Then $c_{\sigma,\Delta}(G)$ is a power of $4$.
\end{theorem}

Let us denote by $\pev(G)$ the minimum of the numbers
$c_{\sigma,\Delta}(G)$, over all choices of $\sigma$ and $\Delta.$ In
view of the relationship of even subgraphs with the Ising model, we
call $\pev(G)$ the {\em Ising complexity} of $G$.

\begin{theorem}[Main Theorem (Cont'd)] \lbl{MT2} For every graph $G$, the Ising complexity
  satisfies $$ \pev(G)=4^g~,$$ where the number $g$ is the embedding
  genus of
$G$.
\end{theorem} 

Here, the {\em embedding genus} of $G$ is the 
  minimal genus of an orientable
  surface in which $G$ can be embedded.

The proof of Theorems~\ref{MT} and~\ref{MT2} will be given in
Section~\ref{sec4}. 

Let us end the Introduction by pointing out some relations of our results with other topics.

\subsection{Equivalence of $\E_G(x)$ and the Ising partition function}
\lbl{sub.eq}
The Ising partition function is defined by 
$$
Z_G^{\mathrm{Ising}}(\beta) = Z^{\mathrm{Ising}}_G(x)\Big|
  _{\textstyle{x_e:=e^{\beta J_e} \ \forall  e\in
  E(G)}}
$$
 where the $J_e$ $( e\in E(G))$ are weights (coupling constants)
 associated with the edges of the graph $G$, the parameter $\beta$ is the inverse
 temperature,  and
$$ 
Z_G^{\mathrm{Ising}}(x)= \sum_{\sigma:V(G)\rightarrow \{1,-1\}} \  \prod_{e= \{u,v\}\in E(G)}x_e^{\sigma(u)\sigma(v)}.
$$
The theorem of van der Waerden \cite{vdW} (see \cite[Section 6.3]{L}
for a proof) states that $Z_G^{\mathrm{Ising}}(x)$ is the same as $\E_G(x)$ up to
change of variables and 
multiplication by a constant factor: 

$$
Z_G^{\mathrm{Ising}}(x)= 2^{|V(G)|}\left(\prod_{e\in E(G)}\frac{x_e+ x_e^{-1}}{2}\right) \E_G(z)\Big|
  _{\textstyle{z_e:= \frac{x_e- x_e^{-1}}{x_e+ x_e^{-1}}}}
$$

\subsection{Determinantal complexity}
\lbl{sub.detcompl}
For a polynomial $P(x_1,\ldots, x_n)$ with rational coefficients, let
the {\em determinantal complexity} of $P$, denoted by $\cdet(P)$, be
the  minimum $m$
so that, if $A$ is the $m\times m$ matrix of variables $(x_{ij})_{i,j= 1,\ldots, m}$, then $P$
may be obtained from the determinant $\det(A)$ by a number of applications of the operation of
replacing some variable $x_{ij}$ by a variable $x_k$ or by a rational constant. This concept was
introduced by Valiant (see \cite{V}) who also proved that $\cdet(P)$ is
at most $2\cform(P)+2$; here $\cform(P)$
denotes the {\em formula size} of $P$, {\em i.e.} the minimum number of additions and multiplications
one needs to obtain $P$ starting from the variables $x_1,\ldots,
x_n$ and constants.  The main  problem in the area of algebraic complexity theory is
to find lower bounds for $\cform(P)$. Lower bounds for $\cdet(P)$ have
recently been investigated extensively (see {\em e.g.} \cite{MR}). We suggest to 
 study $\cdet(\E(G))$ using the methods introduced in the present paper.

\subsection{Pfaffian graphs}
\lbl{sub.pfaf}
It follows from Theorem \ref{MT2} that $\pev(G)=1$ if and only if $G$ is planar. This characterises
the graphs for which the Ising partition function 
$Z_G(x)$ 
is equal to one Pfaffian, in the sense of
Theorems ~\ref{MT} and~\ref{MT2}. 
Note that this characterisation can be formulated in terms of excluded
minors, by Kuratowski's theorem. It also provides a polynomial
algorithm to recognize the graphs $G$ for which $\pev (G)=1$,
since
planar graphs can be recognized in polynomial time.
We remark that it remains a longstanding open problem to characterise 
the {\it Pfaffian graphs}, {\em i.e.} graphs $G$ satisfying 
$\pNor(G)=1$, in a way which yields
a
polynomial recognition algorithm  (see \cite{Th}). 

\subsection{Even drawings} Norine \cite{N} has considered drawings $\varphi$ of a
  graph $G$ on an orientable  surface
  $\Sigma$  such that the self-intersection number of every perfect
  matching $M$ of $G$ in this drawing  is even. (Contrary to our
  definition of drawings (see \ref{2.1} below), he does not, however, allow edges to
  self-intersect.) Let us call a drawing satisfying Norine's definition
  {\em matching-even}. As pointed out by Norine, the Arf invariant
  formula for perfect matchings (Theorem \ref{Pfaff-formula}) goes through if
  we have a matching-even drawing of a graph on a surface in place of
  an embedding. Moreover, Norine has shown that $\pNor(G)=1$ if and only if $G$
has a matching-even drawing in the plane, and $\pNor(G)=4$ if and only if $G$
has a matching-even drawing on  the torus. 
It is, however, conceivable that, in general, the minimal genus of
  a surface supporting a matching-even drawing of $G$ could be smaller
  than the  embedding genus. We will point out  that no such
  phenomenon can occur for even subgraphs (see Section~\ref{sec5} and
  Theorem~\ref{55} for a precise statement).

\section{The Arf invariant formula for perfect matchings}\lbl{sec2}

In this section, we give a proof of the Arf invariant formula for
the perfect matching polynomial.  Other proofs can be found in \cite{GL1,T,CR}.

\begin{definition}\lbl{2.1} A {\em drawing} of 
a finite graph
$G$ on a surface $\Sigma$ is a continuous
  and piecewise
  smooth map 
$\varphi$ from the topological realization of $G$ (as 
a one-dimensional CW-complex) to $\Sigma$
so that $\varphi$ 
is injective
  except for a finite number of transverse double points, subject to
  the condition that for every double point
  $p=\varphi(x)=\varphi(x')$, none of the preimages $x$ and $x'$ is a
  vertex. 
\end{definition}
 In other words, all intersections in the drawing happen in the
 interiors of edges. Note that we allow self-intersections of edges in
 this definition. A drawing without double points is called an
 {\em embedding.} 

If a drawing $\varphi$ of $G$ is given, and $E'\subset E(G)$ is a collection of
edges, we denote by $\kappa_\varphi(E')$ the number $\pmod 2$ of double points
of $\varphi(E')$. 
Note that if 
$E'=\{e_1, \ldots, e_k\}$ is a
collection of distinct edges, then 
\begin{equation}
\lbl{quad1}
\kappa_\varphi(E')=\sum_i \kappa_\varphi(e_i) 
+ \sum_{i<j} 
\cross_\varphi(e_i,e_j)
\ \pmod 2~, 
\end{equation} 
where $\cross_\varphi(e_i,e_j)$ is the number of intersections of the
interiors of the
edges $e_i$ and $e_j$ in the drawing. 
We emphasize that vertices of the graph never count as intersection 
points. 

We will use the following result of Tesler.

\begin{ttheorem}[Tesler \cite{T}]
\lbl{thm.tes}
Let $G$ be a graph drawn in the plane. Then there is 
$\varepsilon_0\in\{\pm 1\}$ and an orientation $D_0$ of $G$ so that for
every perfect matching $M$ of $G$, its sign in  $\Pfaf
(A(G,D_0))$ satisfies 
\begin{equation}\lbl{tes}
\sign(M,D_0)=\varepsilon_0 (-1)^{\kappa_\varphi(M)}~. 
\end{equation}
\end{ttheorem}

\begin{definition}[Tesler] An orientation $D_0$ satisfying (\ref{tes})
   is called a {\em crossing orientation.}
\end{definition}

We now describe how an embedding of a graph in a surface can be used
to make a planar drawing of that graph of a special kind. First,
recall the following standard description of a genus $g$ surface $S_g$
with one boundary component.  (We reserve the notation $\Sigma_g$ for
a closed surface of genus $g$.)

\begin{definition}
\lbl{def.highway}
The {\em highway surface} $S_g$ consists of a base polygon $R_0$ and bridges $R_1, \ldots, R_{2g}$, where 
\begin{itemize}
 \item 
$R_0$ is a convex $4g-$gon with vertices $a_1, \ldots, a_{4g}$ numbered clockwise,
\item
Each $R_{2i-1}$ is a 
rectangle
with vertices $x(i,1), \ldots, x(i,4)$ numbered clockwise. It is glued
with $R_0$ so that its edge $[x(i,1),x(i,2)]$ is identified with the edge $[a_{4(i-1)+1},a_{4(i-1)+2}]$
and the edge $[x(i,3),x(i,4)]$ is identified with the edge $[a_{4(i-1)+3},a_{4(i-1)+4}]$,
\item
Each $R_{2i}$ is a 
rectangle
with vertices $y(i,1), \ldots, y(i,4)$ numbered clockwise. It is glued
with $R_0$ so that its edge $[y(i,1),y(i,2)]$ is identified with the edge $[a_{4(i-1)+2},a_{4(i-1)+3}]$
and the edge $[y(i,3),y(i,4)]$ is identified with the edge
$[a_{4(i-1)+4},a_{4(i-1)+5}]$. (Here, indices are considered modulo $4g$.)
\end{itemize}
\end{definition}

There is an orientation-preserving 
immersion $\Phi$ of $S_g$
into the plane which is injective except that for each $i=1, \ldots
g$, the images of the bridges
$R_{2i}$ and $R_{2i-1}$ intersect in a square. 

Now assume the graph $G$ is embedded into a closed orientable surface $\Sigma_g$
of genus $g$. We think of $\Sigma_g$ as $S_g$ union an additional disk
$R_\infty$ glued to the boundary of $S_g$. By an isotopy of the embedding,
we may assume that $G$ does not meet the disk $R_\infty$, and that,
moreover, all
vertices of $G$ lie in the interior of $R_0$. We may also assume that
the intersection of $G$ with any of the rectangular bridges $R_i$
consists of disjoint straight
lines connecting the two sides 
of $R_i$ which are glued to the base bolygon $R_0$.
(This last assumption is not really needed, but it
makes the proof of Prop.~\ref{2.4} below more transparent.)  If we now compose the embedding of $G$
into 
$S_g$
with the immersion $\Phi$, we get a drawing $\varphi$ of
$G$ in the plane. A planar drawing of $G$ obtained in this way will be
called {\em special}.
Observe that double points of a
special drawing can only come from the intersection of the images of
bridges under the immersion $\Phi$ of $S_g$ into the plane. Thus every double
point of a special drawing lies in one of the squares
$\Phi(R_{2i})\cap \Phi(R_{2i-1})$.

We now explain how a special drawing can be used to get a homological
expression for the sign of a perfect matching. 
Let $H=H_1(\Sigma_g; \BF_2)$ be the first homology group of $\Sigma_g$
with coefficients in the field $\BF_2$. We have canonical isomorphisms
$H\cong H_1(S_g; \BF_2)\cong H_1(S_g,R_0;\BF_2)$. This gives us a
basis 
\begin{equation}\lbl{basis} a_1, b_1, \ldots , a_g,b_g
\end{equation}
 of $H$, where $a_i$ corresponds
to the class of the bridge $R_{2i-1}$ in $H_1(S_g,R_0;\BF_2)$, and
$b_i$ corresponds to the class of $R_{2i}$.
Recall that $H$ has a non-degenerate (skew-)symmetric bilinear form
called the$\pmod 2$ {\em intersection form}. (The name comes from the
fact that this form can be defined
using intersection numbers of closed curves on the surface.) We let
$\cdot$ denote this form.
In the basis (\ref{basis}),  
it 
is given by 
\begin{eqnarray}\lbl{intform}
a_i\cdot a_j=b_i\cdot b_j&=&0 \\
\lbl{intform2}
a_i\cdot
b_j&=&\delta_i^j  
\end{eqnarray}
  for all $i,j=1,\ldots, g$.

\begin{definition}  A {\em quadratic form on $(H,\cdot)$}
is a function $q:H\rightarrow \BF_2$ so
  that 
\begin{equation}\lbl{quad} q(x+y)=q(x)+q(y)+x\cdot y \ \ (x,y \in H)~.
\end{equation}
\end{definition}
 We denote the set of such quadratic forms by $Q$. It follows from
 (\ref{quad}) that $q(0)=0$ for all $q\in Q$. Also, (\ref{quad})
 implies that a quadratic form $q\in Q$ is
determined by its values on a basis of $H$, and these values can be prescribed freely in
$\BF_2$. It follows that $Q$ has $4^g$ elements (the
number of elements in  $\BF_2^{2g}$). Another way to see that $Q$ has $4^g$
elements is to observe that the dual vector
space $H^*=\Hom(H,\BF_2)$ acts simply transitively on $Q$, where $\ell
\in H^*$ acts on $q\in Q$ to give $q+\ell \in Q$.  

The usefulness of special drawings and quadratic forms for studying perfect matchings comes from the
following basic proposition. To state it, let $q_0\in Q$ be the
quadratic form on $H$ whose value on each of the
basis vectors $a_i$ and $b_i$ is zero. 

\begin{proposition} \lbl{2.4} Let $\varphi$ be a special planar drawing of $G$
  obtained from an embedding of the graph $G$ on the surface
  $\Sigma_g$.  Then for every perfect matching $M\subset
  E(G)$ the number of double points $k_\varphi(M)$
  satisfies $$k_\varphi(M)=q_0([M]) \ \pmod 2$$ where $[M]$ is the homology class
  of $M$.
\end{proposition}
Here, the homology class of a perfect matching is defined as
follows. First, since  all
vertices of $G$ lie in the base polygon $R_0$, every edge $e$
of $G$ defines a homology class $[e]$ in
$H_1(S_g,R_0;\BF_2)$. Since this group is canonically identified with
$H$, we can think of $[e]$ as an element of $H$. If now  $M$ is a
collection of distinct edges $e_i$, we let $[M]$ be the sum of the
$[e_i]$.

\begin{proof}[Proof of Prop.~{2.4}.] In view of (\ref{quad1}) and
  (\ref{quad}), it is enough to show that 
\begin{itemize}
\item[(i)] For every edge $e$, the number of double points of
  $\varphi(e)$ is equal to $q_0([e])$ $\pmod 2$.
\item[(ii)] For every pair of distinct edges $e_1, e_2$, the number
$\cross_\varphi(e_1,e_2)$
is equal to $[e_1]\cdot [e_2]$
  $\pmod 2$.  
\end{itemize} 
Recall that every double
point of a special drawing lies in one of the squares
$\Phi(R_{2i})\cap \Phi(R_{2i-1})$.
To prove (i), assume that $e\cap R_{2i-1}$ consists of
$\alpha_i$ straight lines, and  $e\cap R_{2i}$ consists of
$\beta_i$ straight lines, for $i=1,\ldots,g$. Then the number of
double points is $$k_\varphi(e)=\sum \alpha_i \beta_i\ \pmod 2.$$ On the other
hand, the homology class of $e$ is $[e]=\sum_i
\alpha_i a_i +\sum_j \beta_j b_j$. Using  
(\ref{intform}), (\ref{intform2}), and (\ref{quad}),
one has 
\begin{eqnarray}\notag 
q_0([e]) &=& q_0(\sum \alpha_i a_i) + (\sum
\alpha_i a_i)\cdot (\sum \beta_j b_j) + q_0(\sum \beta_j b_j)\\
\notag &=&
\sum
\alpha_i \beta_i \ \pmod 2,
\end{eqnarray}  since $q_0(a_i)=0=q_0(b_j)$ for all $i$ and $j$ 
 by the definition of $q_0$.
Thus $k_\varphi(e)=q_0([e])\pmod 2$, as
asserted. 
Statement (ii) is proved in a similar way.
\end{proof}

The following corollary is immediate from the definition of a crossing
orientation. 

\begin{corollary} Let $\varphi$ be a special planar drawing of $G$
  obtained from an embedding of the graph $G$ on the surface
  $\Sigma_g$. Let $D_0$ be a crossing orientation of $G$ with
  respect to this drawing. Then there is 
$\varepsilon_0\in\{\pm 1\}$  so that for
every perfect matching $M$ of $G$, its sign in  $\Pfaf
(A(G,D_0))$ satisfies $$\sign(M,D_0)=\varepsilon_0 (-1)^{q_0([M])}~.$$
\end{corollary}

Thus the quadratic form $q_0$ controls the sign of any perfect
matching in the orientation $D_0$. The following proposition says
that, more generally, every $q\in Q$ controls the sign of perfect
matchings in some orientation.

\begin{proposition}\lbl{2.6} Let $\varphi$ be a special planar drawing of $G$
  obtained from an embedding of the graph $G$ on the surface
  $\Sigma_g$. Then there is 
$\varepsilon_0\in\{\pm 1\}$, and a collection $(D_q)$ of $4^g$ orientations
of $G$ indexed by quadratic forms $q\in Q$, such that for
every perfect matching $M$ of $G$ one
has  $$\sign(M,D_q)=\varepsilon_0 (-1)^{q([M])}~.$$ 
\end{proposition}
The following notation will be useful: If $D$ is an orientation of a
graph, and $S$ is a set of edges, we write $D(S)$ for the orientation obtained from $D$ by reversing
  the orientation of all edges in $S$.

\begin{proof}[Proof of Prop.~\ref{2.6}.] For $q=q_0$ we take $D_q$ to be the crossing orientation $D_0$ which
  exists by Tesler's theorem (Theorem~\ref{thm.tes}). Any other $q\in
  Q$ can be uniquely written as $q=q_0+\ell$ where $\ell\in H^*$ is a
  linear form on $H$. We define $S_q\subset E(G)$ to be the set
  of edges $e$ such that $\ell([e])\not= 0\in \BF_2$, and define
  $D_q$ to be the orientation $D_0(S_q)$.  We have
\begin{eqnarray}\notag \sign(M,D_q)&=&\sign(M,D_0)\, (-1)^{|M\cap
    S_q|}\\
\notag &=&\varepsilon_0 (-1)^{q_0([M])}\, (-1)^{|\{e\in M | \ell([e])\not=
  0\}|}\\
\notag &=&\varepsilon_0 (-1)^{q_0([M])}\, (-1)^{\ell([M])}\\
\notag &=& \varepsilon_0 (-1)^{q([M])}~,
\end{eqnarray}
as asserted. 
\end{proof}

We now recall the definition of the {\em Arf invariant} of a quadratic form
$q\in Q$. 
Let $N_{0}=2^{g-1}(2^g + 1)$, $N_{1}=2^{g-1}(2^g - 1)$, and observe that $N_0 + N_1=4^g$ and
$N_0 - N_1=2^g$. Recall that any $q\in Q$ is a function $H
  \rightarrow \BF_2$.

\begin{fact} (Arf) Any $q\in Q$ either takes $N_0$ times the
  value $0$ (and hence $N_1$ times the value $1$), or $q$ takes $N_1$
  times the value $0$ (and hence $N_0$ times the value $1$). We define
  $\Arf(q)\in \BF_2$ to be equal to zero in the first case, and equal
  to one in the second case. Thus, for every $q\in Q$ one  has
\begin{equation}\lbl{quad2}
\sum_{x\in H} (-1)^{q(x)} = (-1)^{\Arf(q)} \,2^g~.
\end{equation}
\end{fact}

For more about the Arf invariant, see for example Johnson \cite{J},
Atiyah \cite{At}. We remark that there are $N_0$ quadratic forms of Arf invariant zero,
and (hence) $N_1$ quadratic forms of Arf invariant one. In fact, the
assignment $q\mapsto \Arf(q)$ is itself a quadratic form in an affine
sense (see Theorems 2 and 3 of \cite{At}).

The relevance of the Arf invariant for us comes from the following
Lemma, which is in some sense the dual statement to (\ref{quad2}). 

\begin{lemma}\lbl{2.7} For every $x\in H$, one has 
\begin{equation}\lbl{quad3} \frac{1} {2^g}\sum_{q\in
    Q}(-1)^{\Arf(q)} 
  (-1)^{q(x)}=1 ~.
\end{equation}
\end{lemma}

We defer the proof to the end of this section. Combining Lemma~\ref{2.7}
with Proposition~\ref{2.6}, it follows that for every perfect matching
$M$ of 
$G$,
we have 
\begin{equation}\lbl{arfinv} \frac {\varepsilon_0}{2^g}\sum _{q\in Q}
(-1)^{\Arf(q)}\sign(M,D_q)\ = \ 1~.
\end{equation} We refer to (\ref{arfinv}) as the {\em Arf Invariant
  formula}. Thus, we have obtained the following more precise version
of Theorem~\ref{Pfaff-formula} stated in the introduction.

\begin{ttheorem}[Arf Invariant formula for perfect matchings] \lbl{arfin2} Let the graph $G$ be embedded into a closed orientable
  surface $\Sigma_g$ of genus $g$. 
Then the perfect matching polynomial
$\P_{G}$ can be written as a sum of $4^g$ Pfaffians
associated to 
orientations $D_q$ indexed by
quadratic forms on $H=H_1(\Sigma_g;\BF_2)$:  
$$ \P_{G} = \sum_{q\in Q} \alpha_q\,\Pfaf(A(G), D_q)~,$$  where 
$\alpha_q=  \varepsilon_0
  (-1)^{\Arf(q)}/2^g.$ 
\end{ttheorem}

It remains to give the 
\begin{proof}[Proof of Lemma~\ref{2.7}] For $z\in H$, define $q_z:H\rightarrow \BF_2$ by
  $q_z(x)=q_0(x) + z\cdot x$. 
Since $x\mapsto z\cdot x$ is a linear form on $H$, one has
$q_z\in Q$. We
  claim that $\Arf(q_z)=q_0(z)$. Indeed, one has 
\begin{eqnarray}\notag\sum_{x\in H} (-1)^{q_z(x)}&=&\sum_{x\in H} (-1)^{q_0(x) + z\cdot
  x}\\
\notag &=&(-1)^{q_0(z)}\sum_{x\in H} (-1)^{q_0(x+z)}=(-1)^{q_0(z)}\sum_{x\in H} (-1)^{q_0(x)}\\
\notag &=& (-1)^{q_0(z)}\, (-1)^{\Arf(q_0)} 2^g~,
\end{eqnarray}
 where we have used (\ref{quad}) in the second equality and
(\ref{quad2}) in the last equality. But is is easy to check that
$\Arf(q_0)=0$. This proves the claim that $\Arf(q_z)=q_0(z)$ (again by
the characterization of the Arf invariant in (\ref{quad2})).

Now observe that the correspondence $z\mapsto q_z$ establishes
a bijection $H \mapright{\approx} Q$. This is because $H^*$ acts
simply transitively on $Q$, as already remarked above, and any linear
form $\ell\in H^*$ is of the form $\ell(x)=z\cdot x$ for a unique $z\in H$ (because the intersection form is
non-degenerate). Therefore we can prove (\ref{quad3}) as follows: 
\begin{eqnarray}\notag
\sum_{q\in
    Q}(-1)^{\Arf(q)}(-1)^{q(x)} &=&\sum_{z\in H}(-1)^{\Arf(q_z)}
  (-1)^{q_z(x)}\\
\notag &=& \sum_{z\in H}(-1)^{q_0(z)} (-1)^{q_0(x)+z\cdot x}\\
\notag &=& \sum_{z\in H}
  (-1)^{q_0(z+x)}= \sum_{z\in H}
  (-1)^{q_0(z)}\\
\notag &=&  (-1)^{\Arf(q_0)}\, 2^g =2^g~.
\end{eqnarray} This completes the proof of Lemma~\ref{2.7}. 
\end{proof}

\section{The Arf invariant formula for even subgraphs.}
\lbl{sec3}

Let $G$ be a finite graph. Assume we have chosen a blow-up
$G^\sigma$; recall that $G^\sigma$ is determined by a choice
$\sigma=(\sigma_v)$ of linear orderings of the half-edges
at every vertex $v \in V(G)$. Assume we have also fixed a choice $\Delta=(\Delta_v)_{v \in V(G)}$ of
orientations of the gadgets $\Gamma_v$.

In this section, we begin the proof of Theorem~\ref{MT} by giving an
upper bound for the number $c_{\sigma,\Delta}$ defined in the introduction. This is done by constructing  an embedding of $G$ into an orientable surface which
is compatible with the choice of $\sigma$ and $\Delta$, and then proving an Arf
invariant formula for $\E_G$ coming from this embedding. 

Recall the notion of a  Kasteleyn orientation of a graph $\Gamma$ which is
embedded into the plane equipped with its standard clockwise
orientation. We assume that when we walk around any bounded face of
the embedding, we encounter each edge at most once. This property is
satisfied for $2$-connected graphs, but also for the embeddings of the
gadgets
$\Gamma_v$ that we will consider.

\begin{definition} \lbl{kast} An orientation $D$ of $\Gamma$ is {\em Kasteleyn} if every
bounded face $F$ of the embedding is {\em clockwise odd} with respect
to $D$, meaning that the number of edges $e$ of the boundary of $F$
where the orientation of $e$ in $D$ coincides with the orientation of
$e$ as the boundary of $F$ is odd. 
\end{definition}

If $\Gamma $ is embedded into the interior of an oriented disk $\BD$,
the notion of Kasteleyn orientation is defined in the same way.

\begin{proposition}\lbl{surf} There exists an embedding of $G^\sigma$ into
  a closed oriented surface $\Sigma$ such that each gadget $\Gamma_v$
  is entirely contained in the interior of a closed disk $\BD_v \subset \Sigma$ and the orientation
  $\Delta_v$ is Kasteleyn with respect to the embedding of $\Gamma_v$
  into the disk $\BD_v$. Moreover, the disks $\BD_v$ are pairwise disjoint.
\end{proposition}

\begin{proof}[Proof of Prop.~\ref{surf}.] Consider the gadget $\Gamma_v$ with its chosen orientation
$\Delta_v$. Let $\{e'_1, \ldots e'_d\}$ be the edges of $G^\sigma$
  corresponding to the original (half-)edges of $G$ incident with
  $v$, and let $\Gamma'_v$ be the subgraph of $G^\sigma$ consisting of 
$\Gamma_v$ and these edges. The vertices of $\Gamma'_v$ are those
of $\Gamma_v$ union one vertex, say $u_i$, for each of the edges
$e'_i$ ($i=1,\ldots, d$). We claim that $\Gamma'_v$ can be embedded into an
oriented disk $\BD$ so that 
\begin{itemize}
\item $\Gamma_v$ lies in the interior of the disk,
\item the vertices $u_i$ lie on the boundary of the disk, and
\item the orientation $\Delta_v$ is Kasteleyn with respect to the
  embedding restricted to $\Gamma_v$.
\end{itemize}
To see this, first embed $\Gamma_v$ 
into the interior of the disk
so that $\Delta_v$ is Kasteleyn,
and then add the edges $e'_i$ one after the other so that they never cross
each other.

Thus we obtain a cyclic ordering of the vertices $u_i$, coming from
the orientation of the boundary of the disk. 
It corresponds to a
cyclic ordering $c_v$ of the half-edges incident with $v$ in the original
graph $G$. 
It is important to observe that this cyclic ordering only depends on
$\sigma_v$ and $\Delta_v$. 
The  collection $c=(c_v)_{v\in V(G)}$ of cyclic orderings is sometimes called a rotation system on the graph $G$.
As is well-known, 
$c$  gives $G$ the structure
of a ribbon graph. 
It means that $G$ naturally embeds into an oriented
surface $S(G,c)$ obtained as follows: Take one oriented $d$-gon $P_v$ for every $d$-valent
vertex and one oriented rectangle $I_e\times [0,1]$ for every edge $e$
(here $I_e$ is an interval). Then glue $I_e\times 0$ and $I_e\times 1$
to the boundary of the disjoint union of the polygons in the way prescribed by the structure of the
graph $G$ and the cyclic orientations $c_v$. The surface
$S(G,c)$ has boundary, so we let $\Sigma(G,c)$ be the closed surface
obtained from $S(G,c)$ by gluing disks to the boundary components of
$S(G,c)$.   

By construction, the
blow-up $G^\sigma$ of $G$ also embeds into $\Sigma(G,c)$, {\em via} an
embedding such that
each gadget $\Gamma_v$ is contained in the interior of the polygon
$P_v$. The polygon $P_v$ plays the role of the disk $\BD_v$ in the statement of the
proposition, and $\Delta_v$ is Kasteleyn in $\BD_v$. This completes the proof.
\end{proof}

The genus of the surface $\Sigma(G,c)$ is
  called the {\em genus} of  the ribbon graph $(G,c)$ and denoted by
  $g(G,c)$. It is the minimal genus of a closed orientable surface in which the
  ribbon graph (viewed as the surface $S(G,c)$) embeds.

\begin{definition} \lbl{3.4}  We define
  $g(G,\sigma, \Delta)$ to be $g(G,c)$ where $c$ is constructed from
  $(\sigma, \Delta)$ as in the proof of Prop.~\ref{surf}. It is the
  minimal genus of  a closed orientable surface in which $G^\sigma$
  embeds so that the gadgets $\Gamma_v$ are contained in disjoint
  disks $\BD_v$ and the orientations $\Delta_v$ are Kasteleyn with
  respect to the embeddings $\Gamma_v\subset \BD_v$.
\end{definition}  

Let $g=g(G,\sigma, \Delta)$ be the genus of the surface $\Sigma=\Sigma_g$ obtained in the previous
proposition. 
We now apply the machinery of 
the previous section
to the embedding of
$G^\sigma$ into $\Sigma_g$. 
Decompose $\Sigma_g$ into base polygon $R_0$, bridges
$R_i$, and additional disk $R_\infty$, 
as described in Section~\ref{sec2}. 
Perform  an isotopy of the embedding to make $G^\sigma$ disjoint
from $R_\infty$ and to move all gadgets $\Gamma_v$ entirely into
$R_0$. (This is possible because every gadget $\Gamma_v$ is contained
in its own disk $\BD_v$.) Let $\varphi$ denote the special planar
drawing of $G^\sigma$ obtained  
using the immersion $\Phi$ of the highway surface $S_g=\Sigma_g
\backslash R_\infty$ into the plane. 
Note that in this drawing, the
subgraph consisting of the 
disjoint union of the $\Gamma_v$ is planarly embedded, and
the orientation $\Delta$ of  this subgraph 
is Kasteleyn in the sense of definition~\ref{kast}. 

\begin{lemma} The orientation $\Delta=(\Delta_v)_{v\in V(G)}$ of the
    union of the gadgets $\Gamma_v$ can be extended to a crossing
    orientation $D_0$ of $G^\sigma$ with respect to the drawing $\varphi$.
\end{lemma}

\begin{proof} This follows easily from the construction of a crossing
  orientation in \cite[Section~6]{T}. In fact, the following more
  general statement is true: If we remove from a planar drawing of a
  graph all edges involved in crossings, then any Kasteleyn
  orientation (as defined in \ref{kast}) of the remaining planar graph can be extended to a crossing
  orientation of the original graph. 
\end{proof}

Let $H=H_1(\Sigma_g;\BF_2)$ and let $Q$ be the set of
quadratic forms on $(H,\cdot)$ 
where $\cdot$ is the intersection form on 
$H$.
 Let $D_q=D_0(S_q)$ be the orientations indexed by quadratic
forms $q\in Q$ which were constructed in Prop.~\ref{2.6} starting with
the 
crossing
orientation $D_0$. Recall that $D_0$ corresponds to the quadratic
form $q_0$. 

\begin{proposition}\lbl{adm}  Each $D_q$ is a $\Delta$-admissible
orientation. 
\end{proposition}

\begin{proof} Recall that $D_q$ differs from $D_0$ precisely on the
  set of edges $S_q$ defined as follows: Write $q=q_0+\ell$ where
  $\ell\in H^*$, then  $e\in S_q$ if and only if $\ell([e])\neq 0\in
  \BF_2$. But the edges of the gadgets $\Gamma_v$ are zero in
  homology, since the gadgets are entirely contained in the base
  polygon $R_0$. Thus $S_q\cap E(\Gamma_v)=\emptyset$ for all $v\in
  V(G)$. Hence $D_q$ coincides with $D_0$ on the gadgets. Since $D_0$
  is $\Delta$-admissible by construction, so is every $D_q$. 
\end{proof}

Here is, then, the main result of this section. 

\begin{ttheorem}[Arf invariant formula for even subgraphs (abstract version)]\lbl{3.6} Let $G$ be a finite graph. Choose a blow-up
$G^\sigma$ and an orientation $\Delta$ of the gadgets which replaced
the vertices of $G$ in $G^\sigma$. Let $g=g(G,\sigma, \Delta)$ as
defined in~\ref{3.4}. Then the even subgraph polynomial
  $\E_G(x)$ is a linear
  combination of the $4^g$ $\sigma$-projected Pfaffians $F^\sigma
  _{D_q}(x)$ associated to
  the $\Delta$-admissible 
orientations $D_q$ indexed by
quadratic forms on $H=H_1(\Sigma_g;\BF_2)$:  
$$ \E_{G}(x) = \sum_{q\in Q} \alpha_q F^\sigma
  _{D_q}(x)~,$$  where 
$\alpha_q=  \varepsilon_0
  (-1)^{\Arf(q)}/2^g$, $\varepsilon_0\in \{\pm 1\}$ is the
  universal sign
 coming with the crossing orientation $D_0$,  and $$F^\sigma
  _{D_q}
=\Pfaf \, A(G^\sigma, D_q)\Big|
  _{\textstyle{x_e=1 \ \forall  e\in
  E(G_{}^\sigma) \backslash E(G)}}$$
\end{ttheorem}
\begin{proof} This follows from Formula~(\ref{E=P}) relating $\E_{G}$
  to $\P_{G^\sigma}$, Theorem~\ref{arfin2} applied to $\P_{G^\sigma}$,
  and Prop.~\ref{adm}.
\end{proof}

The following corollary is a more precise version of
Theorem~\ref{Pf2} in the introduction.

\begin{corollary}[Arf invariant formula for even subgraphs (embedded version)]\lbl{Pf2delta}If $G$
embeds into an orientable surface $\Sigma$ of genus $g$, then one
can choose the blow-up  $G^\sigma$ in such a way that there exist 
$4^g$ orientations $D_i$ of $G^\sigma$ 
such that the
even subgraph
polynomial $\E_G(x)$ can be
expressed as a linear combination  of the $\sigma$-projected Pfaffians 
$F_{D_i}^\sigma(x)$
($i=1,\ldots,4^g$.)
 Moreover, for every $v\in V(G)$, each of the orientations $D_i$ induces
the same orientation on the gadget $\Gamma_v$. 
\end{corollary}

\begin{proof}  This follows from Theorem~\ref{3.6} using the fact that
  given an embedding of $G$ into an orientable surface $\Sigma$ of
  genus $g$, we can choose
  $\sigma$ and $\Delta$ in such a way that $g(G,\sigma,\Delta)\leq
  g$. Here is a proof of this fact. Choose an orientation of the surface $\Sigma$. Since $G$ is
  embedded in $\Sigma$, the orientation of $\Sigma$ induces, at every
  vertex $v\in V(G)$,  a cyclic ordering
  $c_v$ of the half-edges incident with $v$. Now construct the graph
  $G^\sigma$ by choosing a linear
  ordering $\sigma_v$ at each vertex $v$ which induces this cyclic
  ordering $c_v$. Then it is easy to see that $G^\sigma$ also embeds
  into $\Sigma$, with each gadget $\Gamma_v$ being embedded into a little
  disk neigborhood $\BD_v$ of $v$ in $\Sigma$. Next, choose the orientations
  $\Delta_v$ of $\Gamma_v$ so that they are Kasteleyn with respect to
  the embeddings of the $\Gamma_v$ into the oriented disks $\BD_v$. Then the surface
  (with boundary) 
  $S(G,c)$ constructed  in the proof of Prop.~\ref{surf} can be
  embedded into $\Sigma$. By the classification of surfaces, it
  follows that the genus of $\Sigma$ is
  greater or equal to $g(G,\sigma,\Delta)$, since $g(G,\sigma,\Delta)$ is the genus of the closed surface
  $\Sigma(G,c)$ obtained by gluing disks to the boundary components of
  $S(G,c)$.
\end{proof}

\begin{remark} An even subgraph $E'\subset E(G)$ can naturally be
  viewed as a $1$-cycle$\pmod 2$ of $G$, and hence defines a homology
  class in $H_1(G; \BF_2)$. Let $[E']$ be the image of this homology
  class in $H=H_1(\Sigma_g;\BF_2)$ under the embedding of $G$ into
  $\Sigma_g$ constructed in  Prop.~\ref{surf}. If now $E'$ corresponds to a perfect
matching $M$ of $G^\sigma$ under the bijection of
Prop.~\ref{thm.delta}, then the homology classes $[E']$ and $[M]$
in $H$ coincide. (This is because every edge in $E(G^\sigma)\backslash E(G)$
is entirely contained in the base polygon $R_0$, and hence zero in
homology.) Therefore, using  Prop.~\ref{2.6}, the $\sigma$-projected Pfaffian polynomial $F^\sigma
  _{D_q}$ can be written
$$F^\sigma
  _{D_q}(x)=\sum_{E'\in \E(G)} \sign(E', F^\sigma
  _{D_q}) \prod_{e\in E'}x_e~,$$ where $\sign(E', F^\sigma
  _{D_q})=\varepsilon_0 (-1)^{q([E'])}~.$ 
\end{remark}

\section{Optimality of the Arf invariant formula}\lbl{sec4}

We now give the proof of Theorem~\ref{MT}. Since Theorem \ref{3.6}
already gives an upper bound for $c_{\sigma, \Delta}(G)$, it remains
only to prove the following.

\begin{ttheorem}\lbl{4.1} Let $G$ be a finite graph. Choose a blow-up
$G^\sigma$ and an orientation $\Delta$ of the gadgets which replaced
the vertices of $G$ in $G^\sigma$. Let $g=g(G,\sigma, \Delta)$ as
defined in \ref{3.4}. Assume there exists $k\geq 1$ and a
collection of 
$\Delta$-admissible orientations $D_i$ and coefficients $\lambda_i\in
\BQ$ ($i=1,\ldots, k$) such that the even subgraph polynomial
  $\E_G(x)$ can be expressed as $$ \E_{G}(x) = \sum_{i=1}^k \lambda_i F^\sigma
  _{D_i}(x)~.$$  Then $k\geq 4^g$. 
\end{ttheorem}

\begin{proof} 
A
  $\Delta$-admissible orientation differs from the crossing
  orientation $D_0=D_{q_0}$ only on edges of the original graph
  $G$. Let $S_i\subset E(G)$ be the set of edges where $D_i$ differs
  from $D_0$. The
  sign of an even subgraph $E'$ in $F^\sigma
  _{D_i}(x)$ is 
\begin{eqnarray}\lbl{Si}
\sign(E', F^\sigma
  _{D_i})&=&\sign(E', F^\sigma
  _{D_0}) (-1)^{|E' \cap S_i|}\\
\notag &=&\varepsilon_0 (-1)^{q_0([E'])}(-1)^{\ell_i(E')}~,
\end{eqnarray}
where we have defined 
\begin{equation}\lbl{Sii}\ell_i(E')=|E' \cap S_i| \pmod 2~.
\end{equation} 

Now recall that any even subgraph $E'\subset E(G)$ can naturally be
  viewed as a  $1$-cycle$\pmod 2$ of $G$, and every 
  $1$-cycle uniquely arises in this way. This establishes an identification 
$$\E(G) \cong C_1(G;\BF_2)~,$$ where $C_1(G;\BF_2)$ is the space of 
$1$-cycles on $G$. Hence $\E(G)$ is
naturally endowed with the structure of an $\BF_2$-vector
space,
called the {\em cycle space} of $G$ in graph theory. Moreover, addition$\pmod 2$ of $1$-cycles corresponds to taking symmetric
difference of even subgraphs. The function $\ell_i$
defined in  (\ref{Sii}) is a linear form $\ell_i$ on
this vector space, since 
$$|(E_1\Delta E_2)\cap S_i| = |E_1 \cap S_i| + |E_2 \cap S_i| \pmod
2~.$$ (Here, $\Delta$ denotes symmetric difference.)

Next, we observe that by the construction of the surface $\Sigma_g$ in
Prop.~\ref{surf}, the composite map $$C_1(G;\BF_2)\rightarrow H_1(G;\BF_2)\rightarrow H_1(\Sigma_g;
\BF_2)$$ induced by the embedding of $G$ into $\Sigma_g$, is onto. In other words, any homology class $x\in H=H_1(\Sigma_g;\BF_2)$
can be realized by some even subgraph of $G$. Choose a
sub-vector space $C$ of $C_1(G;\BF_2)$ which maps isomorphically onto
$H$. 
When we think of $C$ as a subset of $\E(G)$, we  denote  $C$ by
$\C$. Clearly, the zero element of $C$ corresponds to the empty
subgraph $\emptyset$ as an element of $\C\subset \E(G)$.

 For every
$i=1,\ldots, k$, we get a
linear form $\ell'_i$ on $H$ defined as
\begin{equation}\lbl{Siii}
\ell'_i(x)=\ell_i(E'_x) \ \ (x\in H)~,
\end{equation}
 where $E'_x\in \C$ is the
unique element of  $\C$ which maps to $x$.  Observe that the homology
class $[E'_x]\in H$ is equal to $x$.

Define the quadratic form $q_i\in Q$ by $q_i=q_0+\ell'_i$. Putting
together (\ref{Si}) and (\ref{Siii}), we have shown the following: For
every $E'\in \C$, and for every $i=1,\ldots, k$, one has
\begin{eqnarray}\lbl{Sv}
\sign(E', F^\sigma_{D_i})&=&\varepsilon_0
(-1)^{q_0([E'])}(-1)^{\ell_i(E')}\\
\notag &=&\varepsilon_0
(-1)^{q_0([E'])}(-1)^{\ell'_i([E'])}\\
\notag &=&\varepsilon_0
(-1)^{q_i([E'])}
~.
\end{eqnarray}

We are now ready to prove  Theorem~\ref{4.1}. By hypothesis, there exists
$\lambda_i\in \BQ $  ($i=1,\ldots, k$) such that $$\sum_{i=1}^k \lambda_i\,
\sign(E', F^\sigma_{D_i})=1$$ for all even sets $E'\in \E(G)$. Since
every homology class $x\in H$ is realized by some $E'_x$ belonging to the set $\C$
for which expression (\ref{Sv}) is valid, it follows
that 
\begin{equation}\lbl{un} \sum_{i=1}^k \lambda_i\,\varepsilon_0\, (-1)^{q_i(x)}=1
  \ \ \ \ (\forall x\in H)
\end{equation} Now recall from Lemma~\ref{2.7} that 
\begin{equation}\lbl{uni}
\sum_{q\in
    Q}\alpha_q 
  (-1)^{q(x)}=1  \ \ \ \ (\forall x\in H),
\end{equation}
 where $\alpha_q=2^{-g} (-1)^{\Arf(q)}$. The following
 Lemma~\ref{nond} implies that $(\alpha_q)_{q\in Q}$ is the unique
 solution  of (\ref{uni}). Since all $\alpha_q\neq 0$, every $q\in Q$
 must appear in (\ref{un}). It follows that $k\geq |Q|=4^g$, as asserted.
\end{proof}

\begin{lemma} \lbl{nond} One has 
$$
\det\big((-1)^{q(x)}\big)_{(q,x)\in Q\times H} \neq 0.
$$
\end{lemma}

\begin{proof} Recall that any $q\in Q$ can be written $q=q_0+\ell$ for
  a unique linear form $\ell\in H^*$. Thus we can describe the matrix in question as $$
\big((-1)^{q_0(x)+ \ell(x)}\big)_{(\ell,x)\in H^*\times H}
$$ Multiplying this matrix on the right by the diagonal
matrix with entries 
$(-1)^{q_0(x)}$ ($x\in H$), we get $$
\big((-1)^{\ell(x)}\big)_{(\ell,x)\in H^*\times H}
$$ In this $4^g\times 4^g$ matrix, the scalar product of
any two rows corresponding to linear
forms $\ell$ and $\ell'$ is $$
\sum_{x\in H}(-1)^{\ell(x)+\ell'(x)}=\left\{\begin{array}{cl} 4^g &
    \text{if} \ \ell = \ell'\\
0& \text{if} \ \ell \neq \ell'\end{array}\right.$$  (Recall that a non-trivial
linear form on an $\BF_2$-vector space takes the value $0$ as many
times as it takes the value $1$.) Thus the matrix is $2^g$ times an
orthogonal matrix. Hence it is non-singular.
\end{proof}

This completes the proof of Theorem~\ref{4.1}, and (hence) of Theorem~\ref{MT}. 

To prove Theorem~\ref{MT2}, it only remains to  show that given a graph
$G$, the minimal
genus $g(G,\sigma,\Delta)$, over all choices of $(\sigma,\Delta)$, is
equal to the embedding genus of $G$. But this was already shown in the proof of
Corollary~\ref{Pf2delta}. Thus Theorem~\ref{MT2} is proved as well.

\section{Even drawings don't help}\lbl{sec5}

\begin{definition} A drawing $\varphi$ of a graph $G$ on a surface
  $\Sigma$ (as defined in~\ref{2.1}) is called {\em
    even} if the number of double points $\kappa_\varphi(E')$ is even
  for every even subgraph $E'$ of $G$.   
\end{definition}

It is easy to see that the proof of the Arf invariant formula for even subgraphs in
Section~\ref{sec3} goes through if we start with an even drawing of
the graph on a surface in
place of an embedding. More precisely, we can replace in
Corollary~\ref{Pf2delta} the embedding of $G$ with an even drawing of
$G$, and the result still holds. However, 
even drawings cannot reduce
the number of $\Delta$-admissible orientations needed to express
$\E_G$ as a linear combination of $\sigma$-projected Pfaffians. This
can be deduced from Theorem~\ref{MT}. The underlying topological
reason is stated in the next theorem.

\begin{theorem}\lbl{55} Let $G$ be a graph. The minimal genus of an orientable
  surface which supports an even drawing of $G$ is equal to the
  embedding genus of $G$.
\end{theorem}

\begin{proof} One can prove this result by
 algebraic-topological arguments using non-degeneracy of the
 intersection form on closed surfaces. We omit that proof but remark
that
 the main idea can be found in the proof of \cite[Lemma~3]{CN}. (We
 thank M. Schaefer for this reference. M. Schaefer has informed us
 that Theorem~\ref{55} also follows from techniques in  \cite{PSS}.) 

Here is a proof of Theorem~\ref{55} in the spirit of the present paper. Assume we have a
 drawing $\varphi$ of $G$ on a surface $\Sigma$ of genus $g$. The orientation of $\Sigma$ induces a cyclic orientation $c_v$ of
 the half-edges at every vertex $v\in
 V(G)$. As in the proof of Corollary~\ref{Pf2delta},
 we can find  $(\sigma,\Delta)$ inducing that cyclic orientation at
 every vertex. If we now assume the drawing is even, the proof of
 Theorem~\ref{3.6} goes through and we can express $\E_G$ as linear
 combination of $4^{g}$ $\sigma$-projected Pfaffians
 associated to orientations which are all $\Delta$-admissible. But by
 our optimality statement in Theorem~\ref{4.1}, we know that one needs
 at least 
 $c_{\sigma,\Delta}(G)=4^{g(G,\sigma, \Delta)}$ such orientations to
 do that. Thus $g\geq g (G,\sigma, \Delta)$. Since $G$ can be embedded in
 the surface of genus    $g (G,\sigma, \Delta)$ constructed in the
 proof of Prop.~\ref{surf}, it follows that $g$ is greater or equal
 than the embedding genus of $G$, as asserted.
\end{proof}

\ifx\undefined\bysame
	\newcommand{\bysame}{\leavevmode\hbox
to3em{\hrulefill}\,}
\fi

\end{document}